\documentclass[11pt,reqno]{amsart}
\usepackage[letterpaper,margin=1in,footskip=0.25in]{geometry}
\calclayout
\usepackage{amssymb}
\usepackage[only,llbracket,rrbracket]{stmaryrd}
\usepackage{mathrsfs}
\usepackage{microtype}
\usepackage{mathtools}
\usepackage{enumitem}
\usepackage[noadjust]{cite}

\PassOptionsToPackage{dvipsnames}{xcolor}
\usepackage{tikz-cd}

\PassOptionsToPackage{pdfusetitle,colorlinks,pagebackref,bookmarksdepth=3}{hyperref}
\usepackage{bookmark}
\hypersetup{citecolor=OliveGreen,linkcolor=Mahogany,urlcolor=Plum}

\newtheorem{theorem}{Theorem}[section]
\newtheorem{corollary}[theorem]{Corollary}
\newtheorem{lemma}[theorem]{Lemma}

\newtheorem{prop}[theorem]{Proposition}
\newtheorem{conj}[theorem]{Conjecture}
\newtheorem{question}[theorem]{Question}

\newtheorem{alphthm}{Theorem}

\theoremstyle{definition}
\newtheorem{definition}[theorem]{Definition}

\newtheorem{cexample}[theorem]{Counterexample}
\newtheorem{rmk}[theorem]{Remark}

\theoremstyle{remark}
\newtheorem{remark}[theorem]{Remark}

\newcommand{\vecx}{\underline{x}}
\newcommand{\vecu}{\underline{u}}
\newcommand{\m}{\mathfrak{m}}
\newcommand{\oM}{\overline{M}}
\newcommand{\Tor}{\text{Tor}}
\newcommand{\n}{\mathfrak{n}}

\begin{document}

\title[Ulrich modules and weakly lim Ulrich sequences do not always exist]{Ulrich modules and weakly lim Ulrich sequences\\ do not always exist}
\author{Farrah C. Yhee}
\address{}
\email{farrah.yhee@gmail.com}
\urladdr{}

\thanks{}
%\subjclass[2020]{Primary 13C14; Secondary 13C13, 13H10, 13H15}

%\keywords{Ulrich modules, weakly lim Ulrich sequences, multiplicity conjectures}

\makeatletter
  \hypersetup{
    pdfsubject=\@subjclass,pdfkeywords=\@keywords
  }
\makeatother

\begin{abstract}
The existence of Ulrich modules for local domains has been a difficult and elusive open question. For over thirty years, it was unknown whether local domains always have Ulrich modules. In this paper, we answer the question of existence for both Ulrich modules and weakly lim Ulrich sequences -- a weaker notion recently introduced by Ma -- in the negative. We construct many local domains in all dimensions $d \geq 2$ that do not have any Ulrich modules. Moreover, we show that when $d = 2$, these local domains do not have weakly lim Ulrich sequences. A key insight in our proofs is the classification of MCM $R$-modules via the $S_2$-ification of $R$. For local domains of dimension $2$, we show that the existence of weakly lim Ulrich sequences implies the existence of lim Ulrich sequences.

\end{abstract}

\maketitle

\section{Introduction}
We provide the first known counterexamples to the existence of Ulrich modules and weakly lim Ulrich sequences for local domains. In the case of Ulrich modules, this answers a difficult question that has been open for over three decades. In the case of weakly lim Ulrich sequences, this resolves the question of existence in a rather surprising turn of events, given that lim Cohen--Macaulay sequences exist for complete local domains of positive characteristic.  See \cite{BHM} \cite{M}.

Ulrich modules were introduced by Bernd Ulrich in 1984 as a means to study the Gorenstein property of Cohen-Macaulay rings \cite{U}. Since then, the theory of Ulrich modules has become a very active area of research in both commutative algebra and algebraic geometry. Ulrich modules have many powerful and broad applications ranging from the original purpose of giving a criteria for when a local Cohen-Macaulay ring is Gorenstein \cite{U} to new methods of finding Chow forms of a variety \cite{ESW} to longstanding open conjectures in multiplicity theory. 

One of the major applications of Ulrich modules is Lech's conjecture on Hilbert--Samuel multiplicities. More specifically, the existence of Ulrich modules for complete local domains implies the following conjecture:

\begin{conj} [Lech's Conjecture \cite{L1}] \label{Lech}
Let $\varphi: (R,\m,k) \rightarrow (S,\n,l)$ be a flat local map between local rings. Then $e_{\m}(R) \leq e_{\n}(S)$.
\end{conj} 

While philosophically reasonable to expect,  Lech's conjecture has been open in virtually all cases for over sixty years. During the 1960's, Lech proved the conjecture in the case where $\text{dim}(R) = 2$ and the case where the closed fiber $S/ \m S$ is a strict complete intersection \cite{L1}, \cite{L2}. These remained the best partial results for the next three decades. The major breakthroughs that followed utilized either Ulrich modules or sequences of modules $\{M_n\}$ that better approximate the Ulrich condition as $n$ gets larger. For example, in the 1990s, Herzog, Ulrich, and Backelin proved that strict complete intersection rings have Ulrich modules \cite{HUB} which implies that Lech holds for strict complete intersection rings. In 2020, Ma proved Lech's conjecture for standard graded domains over perfect fields via the existence of weakly lim Ulrich sequences \cite{M}.

Recently, in \cite{IMW}, Iyengar, Ma, and Walker introduced two new multiplicity conjectures: 

\begin{conj} [\cite{IMW}]
\label{conj:IMW1}
For a local ring $R$, every nonzero R-module M of finite projective dimension satisfies $\ell_R(M) \geq e(R)$. 
\end{conj}

\begin{conj} [\cite{IMW}]
For a local ring $R$, every short complex $F$ supported on the maximal ideal satisfies $\chi_{\infty}(F) \geq e(R),$ where $\chi_{\infty}(F)$ is the Dutta multiplicity.
\end{conj}

Many of the cases established for these two conjectures in \cite{IMW} utilize Ulrich modules and lim Ulrich sequences. 

Historically, the existence of Ulrich modules has been a challenging question. The existing literature is sparse and has mainly explored positive existence results, i.e. classes of rings for which Ulrich modules exist. The major existence results are that Ulrich modules exist for the following classes of rings:

\begin{enumerate}
    \item two-dimensional, standard graded Cohen-Macaulay domains \cite{BHU} 
    \item strict complete intersection rings \cite{HUB}
    \item generic determinantal rings \cite {BRW}
    \item some Veronese subrings of degree $d$ of a polynomial ring in $n$ variables: 
    \begin{itemize}
        \item $n=3$ with $d$ arbitrary and $n=4$ with $d = 2^{\ell}$ in arbitrary characteristic \cite{DHthe}
        \item arbitrary $n$ and $d$ in characteristic $0$ \cite{ESW}
        \item arbitrary $n$ and $d$ for characteristic $p \geq (d-1)n + (n+1)$ \cite{Sa}
    \end{itemize}
\end{enumerate}
Beyond these results, there has been limited progress. In particular, for over thirty years, it has been unknown whether or not (complete) local domains always have Ulrich modules. Finding a counterexample was considered to be a difficult -- if not intractable -- problem because there are essentially no good criteria to test whether or not a ring has an Ulrich module. Directly proving that a ring has no Ulrich modules involves classifying all of its MCM modules and showing that none of the MCM modules are Ulrich. 

In the first half of the paper, we resolve the question of existence of Ulrich modules in the negative. 

\begin{alphthm}
\label{thm: firstmain}
  Ulrich modules do not always exist for complete local domains.
\end{alphthm}

We prove Theorem \ref{thm: firstmain} by constructing complete local domains $R$ of all dimensions $d \geq 2$ whose $S_2$-ification $S$ is a regular local ring. The key ingredient to proving that our counterexamples do not have Ulrich modules is Lemma \ref{s2}, which states that any MCM module over $R$ is an MCM module over $S$. This yields the following intermediary theorem:

\begin{alphthm}
\label{thm: firstmainsub}
Let $(R, \m, k)$ be a local domain. Suppose $R$ has an $S_2$-ification $S$ such that $S$ is a regular local ring. Then every MCM module of $R$ has the form $S^{\oplus h}$. Consequently $R$ has Ulrich modules if and only if $S$ is an Ulrich module of $R$.  
\end{alphthm}

In the second half of the paper, we consider the existence of (weakly) lim Ulrich sequences. In \cite{M}, Ma introduces the weaker notion of a (weakly) lim Ulrich sequence, which are sequences of finitely generated modules that approximate the Ulrich condition, and shows that their existence implies Lech's conjecture.  He then poses the following question: 

\begin{question}
\label{maq}
Does every complete local domain of characteristic $p>0$ with an $F$-finite residue field admit a lim Ulrich sequence, or at least a weakly lim Ulrich sequence? 
\end{question}

A positive answer to Question \ref{maq} in conjunction with the argument for reduction modulo $p > 0$ in \cite{M1} would have resolved Lech's conjecture in the equicharacteristic case. Given that lim Cohen--Macaulay sequences exist for completely local domains of positive characteristic, it is reasonable to ask if (weakly) lim Ulrich sequences exist for these rings. However, in this paper, we answer the question in the negative. 

\begin{alphthm}
\label{thm: secondmain}
Weakly lim Ulrich sequences do not always exist for complete local domains. 
\end{alphthm}

To prove Theorem \ref{thm: secondmain}, we establish important characterizations of (weakly) lim Ulrich sequences for local domains of dimension $2.$ In particular, we first show: 

\begin{alphthm}
\label{thm: secondsub}
Let $(R, \m, k)$ be a local domain of dimension $2$. If $R$ has a weakly lim Ulrich sequence, then $R$ has a lim Ulrich sequence consisting of torsion-free modules.

Moreover, in the case where $R \subseteq S$ is a local module-finite extension of local domains such that $S/R$ has finite length and $S \subseteq \text{frac}(R)$, there exists a lim Ulrich sequence of $R$-modules that is also a lim Cohen--Macaulay sequence of $S$-modules.
\end{alphthm}

Then we prove the following theorem:

\begin{alphthm}
\label{thm: thirdsub}
Let $(R, \m, k)$ be a local domain of dimension $2$. Suppose $R$ has an $S_2$-ification $S$ that is a regular local ring. If $R$ has a weakly lim Ulrich sequence, then $R$ has an Ulrich module. In particular, by Theorem \ref{thm: firstmainsub}, if $R$ has a weakly lim Ulrich sequence, then $S$ is an Ulrich module of $R$.
\end{alphthm}
\subsection*{Acknowledgements}
I am deeply grateful to Melvin Hochster for introducing me to the questions in this paper and for the many helpful discussions, advice, and support. I would also like to thank Takumi Murayama for helpful conversations and for his comments on this manuscript. Finally, I would like to thank David Eisenbud and Bernd Ulrich for their helpful comments. This paper is based upon work supported in part by grants from the National Science Foundation, DGE 1841052, DMS-1401384, and DMS-1902116.

\section{Preliminaries}
In this section, we review the relevant definitions and properties of Ulrich modules and (weakly) lim Ulrich sequences. Throughout this paper, all rings are commutative with a multiplicative identity. All local rings $(R, \m, k)$ include the Noetherian condition. For simplicity, we assume that $k$ is infinite unless explicitly stated otherwise. In particular, we assume the existence of a minimal reduction generated by $d$ elements for the maximal ideal $\m$ of a local ring $R$.

\subsection*{Notation} Let $(R, \m, k)$ is a local ring of dimension $d$. Let $M$ be a finitely generated module over $R$.  Throughout the paper, we use the following notation: 

\begin{itemize}
    \item $\vecx = x_1, \ldots, x_d$
    \item $\ell_R(M)$ is the length of $M$ as a module over $R$. We write $\ell(M)$ when it is clear from the context which $R$ is being used. 
    \item $H_i(\vecx; M)$ is the $i$-th Koszul homology of the module $M$ with respect to  $\vecx.$
    \item $h_i^R(\vecx; M) = \ell_R(H_i(\vecx; M))$ 
    \item $\chi(\vecx;M) = \sum_{i = 0}^{d}(-1)^{i}\ell(H_i(\vecx ; M)) = \sum_{i=0}^d (-1)^i h_i(\vecx; M).$
    \item $\chi_1(\vecx; M) = \sum_{i = 1}^{d}(-1)^{i-1}\ell(H_i(\vecx ; M)) = \sum_{i=1}^d (-1)^{i-1} h_i(\vecx; M).$
    \item $\nu_R(M)$ is the minimal number of generators of $R$. 
    \item $e_R(M)$ is the multiplicity of $M$ with respect to the maximal ideal $\m.$ When $M=R$, we write $e(R).$
\end{itemize}

\subsection{Ulrich Modules}

\begin{definition}[Hilbert-Samuel Multiplicity]
Let $(R, \m, k)$ is a local ring of dimension $d$. Let $M$ be a finitely generated module over $R$. The \emph{Hilbert-Samuel multiplicity} of $R$ with respect to $\m$ is 
\begin{equation*}
e_R(M) \coloneqq d! \lim_{n \to \infty} \frac{\ell(M/m^nM)}{n^d} 
\end{equation*}
where $\text{dim}(M) = d =R/\text{Ann}_R(M).$
\end{definition}

\begin{definition}[MCM]
Let $M$ be a finitely generated module over $(R,m,k)$. Then $M$ is \emph{maximal Cohen-Macaulay} (or MCM) module of $R$ if $\text{depth}_R(M) = \text{dim}(R).$
\end{definition}

We review some useful facts. 

\begin{prop}
\label{prop:basic}
Let $(R, \m, k)$ is a local domain of dimension $d$. Let $M$ be an MCM module over $R$. Then we have
\begin{enumerate}[label=(\alph*)]
\item $e_R(M) = \text{rank}_R(M) \cdot e(R),$
\item $e_R(M) = \ell(M/IM)$, where $I \subseteq \m$ is a minimal reduction of $\m$, and 
\item $e_R(M) \geq \nu_R(M).$
\end{enumerate}
\end{prop}

The statements in Proposition \ref{prop:basic} are standard in the literature. For example, proofs can be found in \cite{DHthe} and \cite{U}.

\begin{definition}
Let $(R, \m, k)$ be a local ring of dimension $d$. Let $M$ be an MCM module over $R$. Then $M$ is an \emph{Ulrich module} if $e_R(M) = \nu_R(M).$ Equivalently, $M$ is an Ulrich module if $\m M = IM$ for any minimal reduction $I\subseteq \m.$
\end{definition}

\begin{lemma}
\label{lem:algext}
Let $(R, \m, k)$ be a local domain containing $k$. Let $L$ be a finite algebraic extension of $k$. Then $S = L \otimes_k R$ is a local ring with maximal ideal $\m S$ and $S$ has an Ulrich module if and only if $R$ has an Ulrich module. \end{lemma}

The proof of Lemma \ref{lem:algext} is standard. We include it below for completeness. 

\begin{proof}
Observe that $S$ is a free module-finite extension of $R$ and that $\m S$ is the maximal ideal of $S$. So any system of parameters for $R$ is a system of parameters for $S$.

Now suppose $N$ is an Ulrich module over $S$. It is clear that any MCM module over $S$ is an MCM module over $R$. We have $e(R) = e(S)$ because the length of $S/(\m S)^t = L \otimes_R (R/\m^t)$ over $S$ is the same as the length of $R/\m^t$ over $R$. Let $[L:k]$ be the degree of the field extension. Then $\nu_R(N) = [L:k]\nu_S(N)$ and we have
\begin{equation*}
    e_R(N) = \text{rank}_R(N) e(R) = [L:k]\text{rank}_S(N)e(R) = [L:k]\text{rank}_S(N)e(S) = [L:k]e_S(N).
\end{equation*}
Then 
\begin{equation*}
\frac{e_R(N)}{\nu_R(N)} = \frac{e_S(N)}{\nu_S(N)} = 1.
\end{equation*}
So $N$ is an Ulrich module of $R$.

On the other hand, if $M$ is an MCM module of $R$, then $S\otimes_R M$ is an MCM module of $S$, and we have $e_R(M) = e_S(S\otimes_R M)$ and $\nu_R(M) = \nu_S(S \otimes_R M)$. So if $M$ is an Ulrich module of $R$, then $S\otimes_R M$ is an Ulrich module of $S$. 
\end{proof}

\subsection{(Weakly) lim Cohen--Macaulay sequences and (weakly) lim Ulrich sequences}

\begin{definition}
Let $(R, \m, k)$ be a local ring of dimension $d$. A sequence of finitely generated nonzero $R$-modules $\{M_n\}$ of dimension $d$ is called \emph{lim Cohen--Macaulay} if there exists a system of parameters $\vecx$ such that for all $i \geq 1$, we have 
\begin{equation*}
    \lim_{n \to \infty} \frac{\ell(H_i(\vecx; M_n))}{\nu_R(M_n)} = 0.
\end{equation*}

A sequence of finitely generated $R$-modules $\{M_n\}$ of dimension $d$ is called \emph{weaky lim Cohen--Macaulay} if there exists a system of parameters $\vecx$ such that
\begin{equation*}
    \lim_{n \to \infty} \frac{\chi_1(\vecx; M_n)}{\nu_R(M_n)} = 0.
\end{equation*}

\end{definition}

\begin{definition}
Let $(R, \m, k)$ be a local ring of dimension $d$. A sequence of finitely generated nonzero $R$-modules $\{M_n\}$ of dimension $d$ is called \emph{lim Ulrich} (respectively, \emph{weakly lim Ulrich}) if $\{M_n\}$ is lim Cohen--Macaulay (respectively, weakly lim Cohen--Macaulay) and 

\begin{equation*}
    \lim_{n \to \infty} \frac{e_R(M_n)}{\nu_R(M_n)} = 1.
\end{equation*}

\end{definition}

Throughout the paper, we use the following proposition due to \cite{BHM} and \cite{M}.
\begin{prop}[\cite{BHM}\cite{M}]
Let $(R, \m, k)$ be a local ring of dimension $d$. 
\begin{enumerate}[label=(\alph*)]
    \item \cite{BHM} If $\{M_n\}$ is a lim Cohen--Macaulay sequence of $R$, then for every system of parameters $\vecx = x_1, \ldots, x_d$, we have 
\begin{equation*}
    \lim_{n \to \infty} \frac {h_i(\vecx;M_n)}{\nu_R(M_n)} = 0
\end{equation*}
where $i\geq 1.$
    \item \cite{M} If $\{M_n\}$ is a weakly lim Cohen--Macaulay sequence of $R$, then for every system of parameters $\vecx = x_1, \ldots, x_d$, we have 
\begin{equation*}
    \lim_{n \to \infty} \frac {\chi_1(\vecx;M_n)}{\nu_R(M_n)} = 0.
\end{equation*}

\end{enumerate}
\end{prop}

\section{Ulrich modules do not always exist for local domains}

\begin{lemma}
\label{s2}
Let $(R, \m, k)$ be a local domain. If $R$ has an $S_2$-ification $S$ that is a local ring, then any MCM module $M$ of $R$ is an MCM module of $S$. 
\end{lemma}

\begin{proof}
Let $M$ be an MCM module over $R$. We want to show that for any $f \in S - R$ and any $m \in M$, there is a well-defined element $f \cdot m \in M$.  Let $W = R - \{0\}$. Since $M$ is MCM, it is torsion-free over $R$ and embeds in $W^{-1}M$. It suffices to show that $f \cdot (m/1) \in M$. Since the height of the ideal $R :_{R} f$ is at least two, there exist $u$ and $v$ in $R :_{R} f$ such that the sequence $u, v$ is a part of a system of parameters for $R$. Since $M$ is MCM, the sequence $u, v$ is a regular sequence on $M$. Then $v \cdot ((uf) \cdot (m/1)) = u \cdot ((vf) \cdot (m/1)) \in vM$ implies that $(vf)\cdot (m/1) \in vM$. Since $M$ is torsion-free over $R$, we have $f\cdot (m/1) \in M.$
\end{proof}

\begin{theorem}
\label{thm:regularulrich}
Let $(R, \m, k)$ be a local domain. Suppose $R$ has an $S_2$-ification $S$ such that $S$ is a regular local ring. Then every MCM module of $R$ has the form $S^{\oplus h}$. Consequently $R$ has Ulrich modules if and only if $S$ is an Ulrich module of $R$ if and only if $IS = \m S$ for any minimal reduction $I$ of $\m.$
\end{theorem}

\begin{proof}
By Lemma \ref{s2}, any MCM module $M$ over $R$ is MCM over $S$. But $S$ is regular. Hence $M \cong S^{\oplus h}$. The second statement follows immediately because $S^{\oplus h}$ is an Ulrich module of $R$ if and only if $S$ is an Ulrich module of $R$. 
\end{proof}

\begin{theorem}
\label{thm:genclasscounter}
Let $S = k[\![\vecx]\!] = k[\![x_1, \ldots, x_d]\!]$ where $d\geq2.$ Let $\vecu = u_1, \ldots, u_d$ be a system for parameters of $S$ such that $I = (\vecu)S$ is not integrally closed. Let $\overline{I}$ be the integral closure of $I$ in $S$. Let $\{g_{\lambda}\}_{\lambda\in \Lambda}$ be an arbitrary collection of elements in $\overline{I}$ and $f \in \overline{I}-I.$ For $1 \leq j \leq d$, let $v_j, w_j$ be elements of the maximal ideal of $ k[\![\vecu]\!]$ that generate a height $2$ ideal in $k[\![\vecu]\!]$ (e.g. one can take powers of distinct elements in $\{u_1, \ldots, u_d\}).$ Define $R$ to be the domain
\begin{equation*}
    R \coloneqq k[\![\vecu]\!][f][v_jx_j, w_jx_j]_{1 \leq j \leq d}[g_{\lambda}]_{\lambda \in \Lambda}.
\end{equation*}
Then $R$ has no Ulrich modules. 
\end{theorem}

%\begin{theorem}
%Let $S = k[\![x, y]\!]$. Let $u,v$ be a system of parameters of $S$ such that $I = (u,v)S$ is not integrally closed. Let $\overline{I}$ be the integral closure of $I$ in $S$. Let $\{g_j\}_{j \in J}$ be an arbitrary collection of elements in $\overline{I}$ and $f \in \overline{I}-I.$ Let $R \coloneqq k[\![u, v]\!][u^ax, v^bx, u^cy, v^dx, f][g_j]_{j \in J}$ where $a, b, c, d$ are positive integers. Then $R$ has no Ulrich modules. 
%\end{theorem}

\begin{proof}
First, notice that $k[\![\vecu]\!] \subset k[\![\vecx]\!]$ is a module-finite extension. Then $R$ is (Noetherian) local and $R \subset k[\![\vecx]\!] $ is a module-finite extension. 

Let $\m_R$ be the maximal ideal of $R$. From the construction of $R$, it is clear that $\vecu = u_1, \ldots, u_d$ is a system for parameters for $R$ and in fact, a minimal reduction of $\m_R$ because all the other adjoined elements are integral over $(\vecu) S$ in $S$ and thus integral over $(\vecu) R$ in $R$.  Then for all $1 \leq j \leq d$, the element $x_j$ is multiplied into $R$ by $v_j$ and $w_j$ which generate a height $2$ ideal in $R$. Thus $x_j$ is in the $S_2$-ification of $R$ for all $1\leq j \leq d.$ But this means that $S = k[\![\vecx]\!]$ is the $S_2$-ification of $R.$

By Theorem \ref{thm:regularulrich}, it suffices to show that $S$ is not an Ulrich module of $R$. But $(\vecu)S \neq \m_RS$ because $f \notin (\vecu)S.$ Thus $R$ has no Ulrich modules. 
\end{proof}

\begin{remark}
Similar constructions can be used to generate counterexamples that are essentially of finite type over fields. For example, see Counterexample \ref{thm:ulrichdoesnotexist}.
\end{remark}

\begin{remark}
In \cite{IMW}, Iyengar, Ma, and Walker consider rings of the form $T = k + J$ where $S = k[\![x, y]\!]$ and $J \subseteq S$ is an ideal primary to $(x, y)S$. If $J$ has a minimal reduction $I = (u, v)S$, then the rings $T$ have the form in Theorem \ref{thm:genclasscounter}. Thus $T = k +J$ has no Ulrich modules if $J \neq IS.$

In the case where $J$ does not have a minimal reduction, we can reduce to the previous case by taking a finite algebraic field extension of $k$ so that $J$ has a minimal reduction, and then apply Lemma \ref{lem:algext}. 
\end{remark}

\begin{cexample}\label{thm:ulrichdoesnotexist}
We give an explicit counterexample using the construction in \ref{thm:genclasscounter}. The local domain $$R = k[x^n, x^{n+1}, x^ny, y^{n}, y^{n+1}, xy^{n}, xy]_{\mathfrak{m}},$$ 
where $\mathfrak{m}$ is the maximal ideal $(x^n, x^{n+1}, x^ny, y^{n}, y^{n+1}, xy^{n}, xy)$, and its completion 
$$ \widehat{R} = k[\![x^n, x^{n+1}, x^ny, y^n, y^{n+1}, xy^n, xy]\!]$$
do not have Ulrich modules for $n\geq 2.$

The argument is essentially the same for $R$ and $\widehat{R}$. We will work with $R$. The $S_2$-ification of $R$ is $k[x,y]_{(x,y)}$ and the ideal $(xy, x^n-y^n)R$ is a minimal reduction for $\m R$. But $$(xy, x^n-y^n)S \neq (x^n, x^{n+1}, x^ny, y^n, y^{n+1}, xy^n, xy)S$$ as ideals in $S$.
\end{cexample}

\begin{rmk}
We can use $R$ in Counterexample \ref{thm:ulrichdoesnotexist} to give a new counterexample to the localization of Ulrich modules, i.e. a local ring $(T,\mathfrak{n},\ell)$ that has an Ulrich module $M$ and a prime ideal $\mathfrak{p} \subseteq T$ such that $M_{\mathfrak{p}}$ is not an Ulrich module over $T_{\mathfrak{p}}$. While a counterexample to localization was first given by Hanes in \cite{DHthe}, the following counterexample is stronger in the sense that $T$ localizes to a ring that has no Ulrich modules whereas Hanes's counterexample localizes to a ring that does have an Ulrich module.
\end{rmk}

\begin{cexample}[Localization] 
\label{ex:counter}
Consider the ring
\begin{equation*}
    T = k[s^{n+1}, sx^{n}, x^{n+1}, x^ny, sy^n, y^{n+1}, xy^n, s^{n-1}xy]_{\mathfrak{n}}
\end{equation*}
where $\mathfrak{n} = (s^{n+1}, sx^{n}, x^{n+1}, x^ny, sy^n, y^{n+1}, xy^n, s^{n-1}xy)$ and $n\geq2$. Let $$\varphi: T \hookrightarrow k[s,x,y]_{(s,x,y)}$$ be the inclusion map and $\mathfrak{p} = \varphi^{-1}((x,y))$. Then the localization  $T_{\mathfrak{p}}$ is the ring 
\begin{equation*}
  k(s^{n+1})\Big[\Big(\frac{x}{s}\Big)^n, \Big(\frac{x}{s}\Big)^{n+1}, \Big(\frac{x}{s}\Big)^n \Big(\frac{y}{s}\Big), \Big(\frac{y}{s}\Big)^n, \Big(\frac{y}{s}\Big)^{n+1}, \Big(\frac{x}{s}\Big)\Big(\frac{y}{s}\Big)^n, \Big(\frac{x}{s}\Big)\Big(\frac{y}{s}\Big) 
    \Big]
\end{equation*}
localized at the obvious maximal ideal. But $T_{\mathfrak{p}}$ has no Ulrich modules by Theorem \ref{thm:ulrichdoesnotexist}.

It remains to show that $T$ has an Ulrich module. Let $S = k[s,x,y]_{(s, x, y)}^{(n+1)}$ with maximal ideal $\mathfrak{a}$. One can compute $e_{T}(T) = (n+1)^2 = e_{S}(S)$. The rings $T$ and $S$ have the same fraction field and so $\text{rank}_T(S) = 1$. Now $S$ has an Ulrich module $M$ by Proposition 3.6 in \cite{DH2004}. Then $M$ is MCM over $T$ and $\text{rank}_{T}(M) = \text{rank}_{S}(M)$. Then 
\begin{equation*}
(n+1) \cdot \text{rank}_{T}(M) = e_{T}(M) \geq \nu_{T}(M) \geq \nu_{S}(M) = e_{S}(M) = (n+1) \cdot \text{rank}_{T}(M). 
\end{equation*}
Thus $e_{T}(M) = \nu_{T}(M)$ and $M$ is Ulrich over $T.$
\end{cexample}

\section{(Weakly) lim Cohen--Macaulay and (weakly) lim Ulrich sequences over domains of dimension 2 }

\begin{definition}
Let $(R, \m, k)$ be a local ring. Let $\mathcal{M} = \{M_n\}$ be a sequence of nonzero finitely generated $R$-modules. Let $\nu_R(M_n)$ be the minimal number of generators of $M_n$. Let $\{a_n\}$ and $\{b_n\}$ be a sequence of positive integers. We define $\sim_{\mathcal{M}}$ to be the equivalence relation $\sim_{\mathcal{M}}$ where $\{a_n\} \sim_{\mathcal{M}} \{b_n\}$ if $$\lim_{n \to \infty} \frac{a_n -b_n}{\nu_R(M_n)} = 0.$$
For the sake of simplicity, we write $a_n \sim_{\mathcal{M}} b_n$ instead of $\{a_n\} \sim_{\mathcal{M}} \{b_n\}$.
\end{definition}

\begin{lemma}
\label{lem:sameequiv}
Let $(R, \m, k)$ be a local ring. Let $\mathcal{M} = \{M_n\}$ and $\mathcal{N} = \{N_n\}$ be two sequences of nonzero finitely generated $R$-modules. Let $\{a_n\}$ and $\{b_n\}$ be a sequence of non-negative integers. Suppose $\nu_R(N_n) \sim_{\mathcal{M}} \nu_R(M_n).$ If $a_n \sim_{\mathcal{M}} b_n$, then $a_n \sim_{\mathcal{N}} b_n.$ In particular, $\nu_R(M_n) \sim_{\mathcal{N}} \nu_R(N_n).$

\end{lemma}

\begin{theorem}
\label{thm:torsionreduction}

Let $(R, \m, k)$ be a local domain of dimension $2$. Let $\{M_n\}$ be a weakly lim Cohen--Macaulay (resp. weakly lim Ulrich) sequence over $R$. Let $C_n \subseteq M_n$ be a torsion submodule such that the quotient $\oM_n \coloneqq M_n/C_n$ has no finite length submodules. Then the sequence $\{\oM_n\}$ is a lim Cohen--Macaulay (resp. lim Ulrich) sequence over $R$. 
\end{theorem}

\begin{proof} Let $\mathcal{M} \coloneqq \{M_n\}$ be a weakly lim Cohen--Macaulay sequence over $R$. Let $I = (\vecx)$ be a system of parameters of the maximal ideal $\m$ and let $\nu_R(M_n)$ be the minimal number of generators of $M_n$. First, we check that 
\begin{equation*}
 \nu_R(M_n) \sim_{\mathcal{M}} \nu_R(\oM_n).   
\end{equation*}
Consider the short exact sequence
\begin{equation*}
%\label{torsionses}
0 \to C_n \to M_n \to \oM_n \to 0.
\end{equation*}

\vspace{1em}

\noindent We know that $\nu_R(\oM_n) \leq \nu_R(M_n) \leq \nu_R(\oM_n) + \nu_R(C_n).$ So it suffices to show that $$\nu_R(C_n) \sim_{\mathcal{M}} 0.$$
From the short exact sequence, we get the long exact sequence of Koszul homology
\begin{align*}
    0 &\to H_2(\vecx;C_n) \to  H_2(\vecx;M_n) \to H_2(\vecx; \oM_n)\\ 
    &\to H_1(\vecx;C_n) \to H_1(\vecx;M_n) \to H_1(\vecx; \oM_n) \to H_0(\vecx;C_n) \to H_0(\vecx;M_n) \to H_0(\vecx; \oM_n) \to 0.
\end{align*}
Now $\oM_n$ has no finite length torsion submodules, so $H_2(\vecx;\oM_n) = 0$. We observe the following: 

\begin{enumerate}[label=(\alph*)]
    \item $H_2(\vecx;C_n) \cong H_2(\vecx;M_n)$ 
    \item $h_1(\vecx; C_n) \leq h_1(\vecx; M_n)$
    \item $\chi_1(\vecx;M) \geq 0$ for any finitely generated $R$-module $M$ \cite{S}
    \item $\chi(\vecx; C_n) = 0$
    \item $0 \leq h_0(\vecx; M_n) - h_0(\vecx; \oM_n) \leq h_0(\vecx; C_n)$
\end{enumerate}
\vspace{1em}
From (a), (b), and (c), it follows that
\begin{equation*}
0 \leq \chi_1(\vecx; C_n) \leq \chi_1(\vecx; M_n).
\end{equation*}

\vspace{1em}
\noindent and because $\mathcal{M}$ is weakly lim Cohen--Macaulay, we have
\begin{equation}
\label{eq:mid}
    \chi_1(\vecx;C_n) \sim_{\mathcal{M}} 0.
\end{equation}
But $\chi(\vecx;C_n) = 0$ and so, $\chi_1(\vecx;C_n) = h_0(\vecx;C_n) = \ell(C_n/(\vecx)C_n).$ Then the inequality

\begin{equation*}
    \nu_R(C_n) = \ell(C_n/ \m C_n) \leq \ell(C_n/(\vecx)C_n)
\end{equation*}
yields
\begin{equation*}
    \nu_R(C_n) \sim_{\mathcal{M}} 0.
\end{equation*}

Next, we show that $\{\oM_n\}$ is a lim Cohen-Macaulay sequence over $R$. We already know that $h_2(\vecx; \oM_n) = 0.$ It remains to show 
\begin{equation*}
    \lim_{n \to \infty} \frac{ h_1(\vecx;\oM_n)}{\nu_R(\oM_n)} = 0.
\end{equation*}
By Lemma \ref{lem:sameequiv}, it is enough to show that 
\begin{equation*}
    \lim_{n \to \infty} \frac{ h_1(\vecx;\oM_n)}{\nu_R(M_n)} = 0.
\end{equation*}
because $\nu_R(\oM_n) \sim_{\mathcal{M}} \nu_R(M_n)$.
Take the alternating sum of the lengths of the Koszul homology in the exact sequence
\begin{align*}
0 \to H_1(\vecx;C_n) \to H_1(\vecx;M_n) \to H_1(\vecx; \oM_n) \to H_0(\vecx;C_n) \to H_0(\vecx;M_n) \to H_0(\vecx; \oM_n) \to 0.
\end{align*}
This is the sum
\begin{equation*}
    h_1(\vecx; C_n) - h_1(\vecx; M_n) + h_1(\vecx; \oM_n) - h_0(\vecx; C_n) + h_0(\vecx; M_n) - h_0(\vecx;\oM_n) = 0.
\end{equation*}
Then 
\begin{align*}
    h_1(\vecx; \oM_n)& =  -h_1(\vecx; C_n) + h_0(\vecx; C_n) + h_1(\vecx; M_n)  - h_0(\vecx; M_n) + h_0(\vecx;\oM_n)\\
    & = -h_2(\vecx;C_n) + h_1(\vecx;M_n) -h_0(\vecx;M_n) + h_0(\vecx;\oM_n) \\
    &= -h_2(\vecx;M_n) + h_1(\vecx;M_n) -h_0(\vecx;M_n) + h_0(\vecx;\oM_n) \\ 
    & = \chi_1(\vecx;M_n) - (h_0(\vecx;M_n) - h_0(\vecx;\oM_n)). 
\end{align*}
Now we know that 
\begin{equation*}
\chi_1(\vecx;M_n) \sim_{\mathcal{M}} 0
\end{equation*}
and by (e) and \ref{eq:mid} above, we have 
\begin{equation*}
   0 \leq h_0(\vecx;M_n) - h_0(\vecx;\oM_n) \leq h_0(\vecx;C_n) = \chi_1(\vecx;C_n) \sim_{\mathcal{M}} 0.
\end{equation*}
Thus
\begin{equation*}
       \lim_{n \to \infty} \frac{ h_1(\vecx;\oM_n)}{\nu_R(M_n)} = 0
\end{equation*}
and the sequence $\{\oM_n\}$ is lim Cohen--Macaulay. It remains to check that 
\begin{equation*}
      \lim_{n \to \infty} \frac{e_R(\oM_n)}{\nu_R(\oM_n)} = 1.
\end{equation*}
But $e_R(\oM_n) = e_R(M_n)$ and $\nu_R(\oM_n) \sim_{\mathcal{M}} \nu_R(M_n)$, so the condition immediately follows and thus $\{\oM_n\}$ is a lim Ulrich sequence of $R$. 
\end{proof}

\begin{definition}
Let $(R, \m, k)$ be a local domain and let $M$ be finitely generated torsion-free $R$-module. Let $(S, \mathfrak{n}, \ell)$ be a local module-finite extension domain of $R$. Suppose $\mathcal{K} = frac(R) = frac(S).$ Then we define $MS$ to be the $S$-module generated by $M$ in $M \otimes_R \mathcal{K}.$ 
\end{definition}

\begin{remark}
In the case where $R$ is a local domain with a $S_2$-ification $S$ that is local, if $M$ is an MCM module of $R$,  then $MS = M$ by Lemma \ref{s2}.
\end{remark}

\begin{lemma}
\label{thm:partialS2}
Let $(R, \m, k)$ be a local domain of dimension $2$ and let $M$ be a finitely generated torsion-free $R$-module. Let $(S, \mathfrak{n}, \ell)$ be a local module-finite extension domain of $R$. Suppose $S \subseteq \text{frac}(R)$ and $S/R$ has finite length. Choose a fixed constant $t$ such that $\m^tS \subseteq R.$ Let $x,y$ be a system of parameters for $R$. Then 

\begin{enumerate}[label=(\alph*)]
\item $MS \subseteq M:_{\mathcal{K} \otimes_R M}(x^t, y^t)R,$
\vspace{.5em}
\item $(M:_{M \otimes_R \mathcal{K}}(x^t,y^t))/M \cong H_1(x^t,y^t;M),$
\vspace{.5em}
\item $\ell_R(MS/M) \leq h_1(x^t, y^t;M).$
\end{enumerate}
\end{lemma}

\begin{proof} Part (a) is clear by the choice of $t$. Part (c) follows immediately from parts (a) and (b). It remains to prove part (b). Define
\begin{equation*}
\varphi: H_1(x^t,y^t;M) \to (M:_{M \otimes_R \mathcal{K}}(x^t,y^t))/M 
\end{equation*}
to be the map
\begin{equation*}
  [(u, v)] \mapsto \Big[\frac{u}{y^t} \Big] = \Big[\frac{-v}{x^t} \Big]
\end{equation*}
where the equality follows from the relation $ux^t + vy^t = 0.$ 
This map is well-defined. If $[(u,v)]$ is trivial, then there exists $w \in M$ such that $[(u,v)] = [y^tw, -x^tw]$. But $[y^tw, -x^tw]$ is mapped to $[(y^tw)/y^t] = [w/1] =0.$

For the map going the other direction, define 
\begin{equation*}
\psi: (M:_{M \otimes_R \mathcal{K}}(x^t,y^t))/M \to H_1(x^t,y^t;M)
\end{equation*}
to be the map 
\begin{equation*} 
[f] \mapsto [(y^tf, -x^tf)].
\end{equation*} 
This is clearly well-defined. The maps $\varphi$ and $\psi$ are inverses, so we are done.
\end{proof}

\begin{theorem}
\label{thm:torsionfreeoverS}
Let $(R, \m, k)$ be a local domain of dimension $2$. Let $(S, \mathfrak{n}, \ell)$ be a local module--finite extension domain of $R$ such that $S \subseteq \text{frac}(R)$ and $S/R$ has finite length. Let $\mathcal{M} = \{M_n\}$ be a lim Cohen--Macaulay (resp. lim Ulrich) sequence of torsion-free $R$-modules. Then the sequence $\mathcal{N} = \{M_nS\}$ is a lim Cohen--Macaulay (resp. lim Ulrich) sequence of $R$-modules and also a lim Cohen--Macaulay sequence  of $S$-modules. 
\end{theorem}
 
\begin{proof}
We first prove that 
\begin{equation*}
    \nu_R(M_n) \sim_{\mathcal{M}} \nu_R(M_nS).
\end{equation*}
Let $Q_n = M_nS/M_n.$ Note that $Q_n$ has finite length because $S/R$ has finite length. The short exact sequence
\begin{equation*}
\label{eq:partialSES}
    0 \to M_n \to M_nS \to Q_n \to 0
\end{equation*}
yields the long exact sequence 
\begin{equation*}
    \ldots \to \Tor_1^R(Q_n,k) \to M_n \otimes_R k \to M_nS \otimes_R k \to Q_n\otimes_R k \to 0.
\end{equation*}
Then \begin{equation*}
\nu_R(M_n) \leq \nu_R(M_nS) + \ell(\Tor_1^R(Q_n,k)) \leq \nu_R(M_n) + \nu_R(Q_n) + \ell(\Tor_1^R(Q_n,k)).
\end{equation*}
and so it suffices to show that
\begin{equation*}
\ell(\Tor_1^R(Q_n,k))\sim_{\mathcal{M}}0 \text{\hspace{1em} and \hspace{1em}} \nu_R(Q_n) \sim_{\mathcal{M}}0.
\end{equation*}
Let $\vecx = x_1, x_2$ be a system of parameters for $R$.   By Lemma \ref{thm:partialS2}, we know that $\ell(Q_n) \leq h_1(x_1^t, x_2^t;M_n)$ for some fixed $t$. But $\mathcal{M}$ is a lim Cohen-Macaulay sequence, so $h_1(x_1^t, x_2^t;M_n) \sim_{\mathcal{M}} 0$ and 
\begin{equation*}
    \ell(Q_n) \sim_{\mathcal{M}} 0.
\end{equation*} 
Then 
\begin{equation*}
\nu_R(Q_n) = \ell(Q_n/ \m Q_n) \sim_{\mathcal{M}}0.
\end{equation*}
Next, by taking a prime cyclic filtration of $Q_n$, one can observe that 
\begin{equation*}
\ell(\Tor^R_1(Q_n,k)) \leq \ell(Q)\ell(\Tor^R_1(k,k)).
\end{equation*}
Then it immediately follows that 
\begin{equation*}
    \ell(\Tor_1^R(Q_n,k))\sim_{\mathcal{M}}0.
\end{equation*}

\vspace{1em}
We now show that $\mathcal{N} = \{M_nS\}$ is a lim Cohen--Macaulay sequence of $R$-modules. It is enough to show that 
\begin{equation*}
    h_1(\vecx;M_nS) \sim_{\mathcal{M}} 0.
\end{equation*}
Because $M_n$ and $M_nS$ are torsion-free over $R$, we have the long exact sequence 
\begin{align*}
    0 &\to H_2(\vecx;Q_n) \to H_1(\vecx;M_n) \to H_1(\vecx;M_nS) \to H_1(\vecx; Q_n)\\ &\to  H_0(\vecx;M_n) \to H_0(\vecx;M_nS) \to H_0(\vecx;Q_n) \to 0.
\end{align*}
Observe that for all $i \geq 0$ 
\begin{equation*}
    h_i(\vecx; Q_n) \sim_{\mathcal{M}} 0.
\end{equation*}
 We see that
 \begin{equation*}
 h_2(\vecx; Q_n) \sim_{\mathcal{M}} 0
 \end{equation*}
 because $H_2(\vecx;Q_n)$ injects into $H_1(\vecx;M_n)$. We already proved that $\ell(Q_n) \sim_{\mathcal{M}} 0$. It immediately follows that 
\begin{equation*}
h_0(\vecx; Q_n) = \ell(Q_n/ \vecx Q_n) \sim_{\mathcal{M}} 0.
\end{equation*}
Then, it follows from $\chi(\vecx; Q_n) = 0$ that 
\begin{equation*}
h_1(\vecx; Q_n) \sim_{\mathcal{M}} 0.
\end{equation*}
From the long exact sequence on Koszul homology, we have 
\begin{equation*}
    h_1(\vecx;M_nS) \leq h_1(\vecx;M_n) + h_1(\vecx;Q_n).
\end{equation*}
But $h_1(\vecx;M_n) \sim_{\mathcal{M}} 0$ and $h_1(\vecx;Q_n)\sim_{\mathcal{M}} 0.$ Therefore, $\mathcal{N} = \{M_nS\}$ is a lim Cohen--Macaulay sequence over $R$.

If $\mathcal{M}$ is lim Ulrich, it immediately follows that $\mathcal{N}$ is lim Ulrich because $e_R(M_n) = e_R(M_nS)$ and $\nu_R(M_n) \sim_{\mathcal{M}} \nu_R(M_nS).$ It remains to check that $\mathcal{N} = \{M_nS\}$ is a lim Cohen--Macaulay sequence for $S.$

For any i, the Koszul homology $H_i(\vecx;M_nS)$ does not change for whether we think of $M_nS$ as an $R$-module or an $S$-module. We also have 
\begin{equation*}
\nu_R(M_nS) \leq \nu_R(S)\nu_S(M_nS),
\end{equation*}
which yields 
\begin{equation*}
 \frac{\nu_R(M_nS)}{\nu_R(S)} \leq \nu_S(M_nS).
\end{equation*}
Then 
\begin{equation*}
    \lim_{n\to \infty} \frac{h_1^S(\vecx;M_nS)}{\nu_S(M_nS)} \leq    \lim_{n\to \infty} \frac{\nu_R(S)h_1^R(\vecx;M_nS)}{\nu_R(M_nS)} =\nu_R(S) \lim_{n\to \infty} \frac{h_1^R(\vecx;M_nS)}{\nu_R(M_nS)} = 0.  
\end{equation*}
\vspace{0.5em}

\noindent Thus $\mathcal{N} = \{M_nS\}$ is a lim Cohen-Macaulay sequence over $S$.
\end{proof}

\begin{theorem}
\label{thm:regularlimUlrich}
A sequence of finitely generated nonzero torsion--free modules $\{N_n\}$ over a regular local ring $S$ of dimension $2$ is lim Cohen--Macaulay if and only if for the minimal free resolution 

\begin{equation*}
0 \to S^{b_n} \to S^{a_n} \to N_n \to 0
\end{equation*} 
we have $ \lim_{n \to \infty} b_n/a_n = 0.$ Such a sequence is always lim Ulrich. 

\end{theorem}

\begin{proof}
Let $x, y$ be a regular system of parameters for $S$. We have
\begin{equation*}
    a_n = \nu_S(N_n) = h_0(x,y; N_n)
\end{equation*}
and 
\begin{equation*}
    b_n = h_1(x,y; N_n). 
\end{equation*}

\noindent Then $\{N_n\}$ is lim Cohen--Macaulay over $S$ if and only if 

\begin{equation*}
\lim_{n \to \infty} \frac{h_1(x,y; N_n)}{\nu_S(N_n)}  = \lim_{n \to \infty} \frac{b_n}{a_n} = 0.
\end{equation*}
Moreover, we have
\begin{equation*}
    \lim_{n \to \infty} \frac{e_S(N_n)}{\nu_S(N_n)} = \lim_{n \to \infty} \frac{a_n - b_n}{a_n} = 1. 
\end{equation*}
Thus $\{N_n\}$ is a lim Ulrich sequence for $S.$
\end{proof}

\section{Weakly lim Ulrich sequences do not always exist for local domains}

\begin{theorem}
\label{thm:main}
Let $(R, \m, k)$ be a local domain of dimension $2$. Suppose $R$ has an $S_2$-ification $S$ that is a regular local ring. The following are equivalent: 

\begin{enumerate}[label=(\alph*)]
    \item $R$ has a weakly lim Ulrich sequence.
    \item $R$ has an Ulrich module. 
    \item $S$ is an Ulrich module of $R$. 
    \item For any minimal reduction $I$ of $\m$, we have $IS = \m S.$
\end{enumerate}

\end{theorem}

\begin{proof} 
First, (c) and (d) are equivalent by definition. Next (b) and (c) are equivalent by Theorem \ref{thm:regularulrich}. It is clear that (b) implies (a). It remains to show that (a) implies (b). 

Suppose $R$ has a weakly lim Ulrich sequence. Then by Theorems \ref{thm:torsionreduction} and \ref{thm:torsionfreeoverS}, there exists a lim Ulrich sequence $\mathcal{M} = \{M_n\}$ of torsion-free $R$ modules that are also $S$ modules. Consider the minimal free resolution 
\begin{equation*}
0 \to S^{b_n} \to S^{a_n} \to M_n \to 0
\end{equation*} 
where $a_n = \nu_S(M_n).$ Now
\begin{equation*}
    \nu_R(S^{b_n}) = b_n\nu_R(S)
\end{equation*}
and 
\begin{equation*}
    e_R(S^{b_n}) = b_n e_R(S).
\end{equation*}
Then Theorem \ref{thm:regularlimUlrich} yields 
\begin{equation*}
   \lim_{n \to \infty} \frac{\nu_R(S^{b_n})}{\nu_R(M_n)} =  \lim_{n \to \infty} \frac{\nu_R(S)b_n}{\nu_R(M_n)} \leq  \lim_{n \to \infty} \frac{\nu_R(S)b_n}{\nu_S(M_n)} = \lim_{n \to \infty} \frac{\nu_R(S)b_n}{a_n} = 0,
\end{equation*}
and 
\begin{equation*}
   \lim_{n \to \infty} \frac{e_R(S^{b_n})}{\nu_R(M_n)} =  \lim_{n \to \infty} \frac{e_R(S)b_n}{\nu_R(M_n)} \leq  \lim_{n \to \infty} \frac{e_R(S)b_n}{\nu_S(M_n)} = \lim_{n \to \infty} \frac{e_R(S)b_n}{a_n} = 0.
\end{equation*}

\vspace{0.5em}
\noindent Consequently, by the minimal free resolution above, we have 
\begin{equation}
\label{eq:first}
    \nu_R(M_n) \sim_{\mathcal{M}} \nu_R(S^{a_n}) = \nu_R(S)a_n,
\end{equation}
and 
\begin{equation}
\label{eq:second}
   e_R(M_n) \sim_{\mathcal{M}} e_R(S^{a_n}) = e_R(S)a_n.
\end{equation} 
Combining equivalences \ref{eq:first} and \ref{eq:second}, we have
\begin{equation*}
\frac{e_R(S)}{v_R(S)} =
\lim_{n \to \infty} \frac{e_R(S)}{v_R(S)} = \lim_{n \to \infty} \frac{e_R(M_n)}{v_R(M_n)} =  1.
\end{equation*}
Thus $S$ is an Ulrich module of $R$. 
\end{proof}

\begin{theorem} 
Weakly lim Ulrich sequences do not always exist for (complete) local domains. 
\end{theorem}

\begin{proof}
This is immediate by Theorem \ref{thm:genclasscounter} by taking $d = 2$ and applying Theorem \ref{thm:main}. 
\end{proof}

\begin{corollary}[Localization]
Weakly lim Ulrich sequences do not always localize for local domains. More precisely, there exist local domains $(R, \m, k)$ that have a weakly lim Ulrich sequence $\{M_n\}$ and a prime ideal $\mathfrak{p}$ such that $\{(M_n)_{\mathfrak{p}}\}$ is not a weakly lim Ulrich sequence for $R_{\mathfrak{p}}.$ Moreover, there exist local domains that have weakly lim Ulrich sequences and a prime ideal $\mathfrak{p}$ such that $R_{\mathfrak{p}}$ has no weakly lim Ulrich sequences.
\end{corollary}

\begin{proof}
This is immediate by taking $k$ to be perfect and $\text{char}(k) >0$ in Counterexample \ref{ex:counter} and applying Theorem \ref{thm:main}.
\end{proof}

%\section{Future Directions}


\begin{thebibliography}{ulrich}

\bibitem[BHM]{BHM} B. Bhatt, M. Hochster, and L. Ma. ``Lim Cohen--Macaulay sequence," in preparation.


\bibitem[BHU]{BHU} J.P. Brennan, J. Herzog, and B. Ulrich. ``Maximally Generated Cohen-Macaulay Modules." Math. Scand. 61(1987), 181-203.

\bibitem[BRW]{BRW} W. Bruns, T. R\"{o}mer, and A. Wiebe. ``Initial algebras of determinantal rings, Cohen-Macaulay and Ulrich ideals." Michigan Math. J. 53(2005), 71-81.

\bibitem[ES]{ESW} D. Eisenbud, F. Schreyer, and Appendix by J. Weyman. ``Resultants and chow forms via exterior syzygies." J. Amer. Math. Soc. 16(2003), no. 3, 537-579.

\bibitem[Ha99]{DHthe} D. Hanes. ``Special Conditions on Maximal Cohen-Macaulay
Modules, and Applications to the Theory of
Multiplicities." Thesis, University of Michigan, 1999.

\bibitem[Ha04]{DH2004} D. Hanes. ``On the Cohen-Macaulay modules of graded subrings." Trans. Amer. Math. Soc. 357(2004), no. 2,  735-756. 

\bibitem[HUB]{HUB} J. Herzog, B. Ulrich, and J. Backelin.  ``Linear maximal Cohen-Macaulay modules over strict complete intersections." J. Pure and Applied Algebra. 71(1991), 187-202.

\bibitem[Hoc17]{Hoc17}
  M. Hochster. ``Homological conjectures and lim Cohen-Macaulay sequences.'' \textit{Homological and computational methods in commutative algebra.} Springer INdAM Ser., Vol. 20. Cham: Springer, 2017, 173–197.
  
\bibitem[IMW]{IMW} S. Iyengar, L. Ma, and M. Walker. ``Multiplicities and Betti Numbers in local algebra via lim Ulrich points," in preparation.

\bibitem[L1] {L1} C. Lech. ``Note on multiplicities of ideals." Ark. Mat. 4(1960), 63–86.

\bibitem[L2] {L2} C. Lech. ``Inequalities related to certain couples of local rings." Acta Math. 112(1964), 69-89.
  
\bibitem[M1]{M1} L. Ma. ``Lech's conjecture in dimension three." Adv. Math. 322(2017), 940-970.

\bibitem[M2]{M} L. Ma. ``Lim Ulrich sequence: proof of Lech's conjecture for graded base rings." May 5, 2020. Submitted. arxiv:2005.02338v1 [math.AC].

\bibitem[S]{S} J.-P. Serre. \emph{Local Algebra}. Springer Monographs in Mathematics. New York: Springer-Verlag Berlin Heidelberg, 2000.

\bibitem[Sa]{Sa} R. C. Saccochi. \emph{Ulrich Schur Bundles in Prime Characteristic p.} Ph.D. thesis. University of Illinois at Chicago, 2020, 70 pp.url:https://www.proquest.com/docview/2519038272.

\bibitem[U]{U} B. Ulrich. ``Gorenstein rings and modules with high numbers of generators." Mathematische Zeitschrift. 188(1984), 23-32.

\end{thebibliography}
\end{document}